\documentclass[11pt]{amsart}

\usepackage{amsthm,amssymb,amsfonts,amsmath,enumerate,graphicx,color}
\newcommand{\C}{\mathbb{C}}

\newcommand{\R}{\mathbb{R}}

\newcommand{\id}{\mathrm{id}}

\newcommand{\des}{\mathrm{des}}

\newcommand{\sgn}{\mathrm{sgn}}

\newcommand{\cone}{\mathrm{cone}}

\newcommand{\todo}[1]{\par \noindent
  \framebox{\begin{minipage}[c]{0.95 \textwidth} TO DO:
      #1 \end{minipage}}\par}

\newcommand\commentout[1]{}

\newtheorem{theorem}{Theorem}[section]

\newtheorem{proposition}[theorem]{Proposition}
\newtheorem{lemma}[theorem]{Lemma}

\theoremstyle{remark}
\newtheorem{example}[theorem]{Example}
\newtheorem{remark}[theorem]{Remark}

\theoremstyle{definition}
\newtheorem{definition}[theorem]{Definition}

\begin{document}

\title[Hyperoctahedral Eulerian Idempotents\ldots]{Hyperoctahedral Eulerian Idempotents, Hodge Decompositions, and Signed Graph Coloring Complexes}

\author{Benjamin Braun}
\address{Department of Mathematics\\
         University of Kentucky\\
         Lexington, KY 40506}
\email{benjamin.braun@uky.edu}

\author{Sarah Crown Rundell}
\address{Department of Mathematics and Computer Science\\
         100 West College Street\\
         Denison University\\
         Granville, OH, 43023}
\email{rundells@denison.edu}

\date{27 July 2013}

\thanks{The first author is partially supported by the National Security Agency through award H98230-13-1-0240.}

% abstract information
\begin{abstract}
Phil Hanlon proved that the coefficients of the chromatic polynomial of a graph $G$ are equal (up to sign) to the dimensions of the summands in a Hodge-type decomposition of the top homology of the coloring complex for $G$.
We prove a type B analogue of this result for chromatic polynomials of signed graphs using hyperoctahedral Eulerian idempotents.
\end{abstract}

\maketitle

%%%%%%%%%%%%%%%%%%%%%%%%%%%%%%%%%%%%%%%%%%%%%%%%%%%%%%%%%%%%%%%%%%
%%%%%%%%%%%%%%%%%%%%%%%%%%%%%%%%%%%%%%%%%%%%%%%%%%%%%%%%%%%%%%%%%%

\section{Introduction}

Let $G$ denote a finite graph and $\chi_G(\lambda)$ its chromatic polynomial.
The coloring complex $\Delta_G$ was defined by Einar Steingr\'{\i}msson \cite{sm} in order to provide a Hilbert-polynomial interpretation of $\chi_G(\lambda)$.
While Steingr\'{\i}msson's original definition of $\Delta_G$ was motivated by algebraic considerations, the coloring complex can also be obtained as the link complex for a hyperplane arrangement, using techniques developed by J\"{u}rgen Herzog, Vic Reiner, and Volkmar Welker \cite{hrw}.
Coloring complexes have many interesting properties.
Jakob Jonsson proved \cite{jo} that $\Delta_G$ is homotopy equivalent to a wedge of spheres in fixed dimension, with the number of spheres being one less than the number of acyclic orientations of $G$.
Axel Hultman \cite{hu} proved that $\Delta_G$, and in general any link complex for a sub-arrangement of the type A or type B Coxeter arrangement, is shellable.
Further, $\Delta_G$ admits a convex ear decomposition, as shown by Patricia Hersh and Ed Swartz \cite{hs}.

A fascinating result due to Phil Hanlon \cite{ha} is that (up to sign) the $j$-th coefficient of $\chi_G(\lambda)$ is equal to the dimension of the $j$-th summand in a Hodge-type decomposition of the top homology of $\Delta_G$.
Hanlon's Hodge decomposition is obtained using the Eulerian idempotents in the group algebra of the symmetric group $S_n$.
These are the elements $e_n^{(j)}\in \C[S_n]$ defined by 
\[
e_n(x)=\sum_{j=1}^nx^je_n^{(j)} = \sum_{\pi \in S_n}\binom{x+n-\des(\pi)-1}{n}\sgn (\pi)\pi \, .
\]
These first arose in work of various authors in the 1980's.
Murray Gerstenhaber and Samuel Schack \cite{gs} proved that all splitting sequences for Hochschild homology arise as linear combinations of Eulerian idempotents, which in certain cases coincides with Hodge decompositions for smooth compact complex varieties; similar work was independently introduced by Jean-Louis Loday \cite{lo}.
A nice introduction to these results can be found in another paper due to Hanlon \cite[Section 1]{hahodge}.
The chain complex defining Hochschild homology is quite similar to the chain complex for $\Delta_G$, and thus Hanlon was able to adapt the Eulerian idempotent splitting techniques in Hochschild homology to produce a similar decomposition for the top homology of $\Delta_G$.

The coloring complex construction can be extended to hypergraphs, and Eulerian idempotents continue to play a role in this setting.
Combinatorial and topological properties of hypergraph coloring complexes were investigated by Felix Breuer, Aaron Dall, and Martina Kubitzke \cite{bdk}, who found that many of the nicest properties of graph coloring complexes are lost in the transition to hypergraphs, e.g. Cohen-Macaulayness, partitionability, being a wedge of spheres, etc.
Hypergraph coloring complexes were also considered  by Jane Long and the second author \cite{lr}.
They show that the homology of hypergraph coloring complexes admits a Hodge decomposition induced by Eulerian idempotents, and that the coefficients of the chromatic polynomial of a hypergraph are essentially the Euler characteristics of the Hodge subcomplexes, up to sign.
The second author \cite{r1} investigated the special case of $k$-uniform hypergraphs, showing that their coloring complexes are shellable and that their cyclic coloring complexes have a certain homology group whose dimension is given by a binomial coefficient.

The Eulerian idempotents play key roles in other contexts as well.
Adriano Garsia \cite{ga} and Christophe Reutenauer \cite{re} studied Eulerian idempotents in their work on free Lie algebras.
Persi Diaconis and Jason Fulman \cite{df} show that the Eulerian idempotents are (up to the sign involution) eigenvectors of an ``amazing'' matrix arising from the study of ``carries'' in addition algorithms.
They also show that this matrix is related to the Veronese construction in commutative algebra.
Phil Hanlon and Patricia Hersh \cite{hh} prove that the homology of the complex of injective words admits a Hodge decomposition, where the dimension of the $k$-th Hodge summand is equal to the number of derangements with exactly $k$ cycles.
These and other results are all-the-more fascinating due to their type B extensions.
The type B Eulerian idempotents, defined in Section~\ref{sec:eulerian}, were originally defined by Fran{\c{c}}ois Bergeron and Nantel Bergeron \cite{bb}.
They proved type B extensions of several of the type A results given above.
The Eulerian idempotents in types A and B also play an interesting role in shuffling problems, as discussed in several of the papers just referenced.

Given the variety of interesting applications of Eulerian idempotents, we believe that a type B version of Hanlon's result regarding $\chi_G(\lambda)$ is of interest.
The goal of this paper is to prove Theorem~\ref{thm:chromhodge}, which provides the desired extension in the setting of signed graph chromatic polynomials.
Section~\ref{sec:signed} contains a review of basic properties of signed graphs and their chromatic polynomials.
Section~\ref{sec:complex} discusses signed graph coloring complexes and hyperoctahedral group actions on them.
In Section~\ref{sec:eulerian}, we prove that the type B Eulerian idempotents induce a Hodge-type decomposition on the top homology of each signed graph coloring complex.
In Section~\ref{sec:chromatic} we prove our main result.

%%%%%%%%%%%%%%%%%%%%%%%%%%%%%%%%%%%%%%%%%%%%%%%%%%%%%%%%%%%%%%%%%%

\section{Signed graphs and chromatic polynomials}\label{sec:signed}

This section is based on Zaslavsky's papers \cite{zas1, zas2}.

\begin{definition}
A \emph{signed graph} $G$ on the vertex set $[n]$ is a multiset $E$ of one-element subsets of $[n]$, called \emph{half-edges}, and two-element subsets of $[n]$, called \emph{edges}, together with a \emph{sign map} $\sigma:E\cap\binom{[n]}{2}\rightarrow \{1,-1\}$ such that $\sigma^{-1}(1)$ and $\sigma^{-1}(-1)$ are each the edge set of a simple graph on $[n]$.
For an edge $e\in E$, if $\sigma(e)=1$, then $e$ is called a \emph{positive edge} of $G$.
If $\sigma(e)=-1$, then $e$ is called a \emph{negative edge} of $G$.
\end{definition}

\begin{example}\label{ex:main}
Let $G$ have vertex set $\{1,2,3\}$, positive edge $\{1,2\}$, negative edges $\{1,2\}$ and $\{2,3\}$, and half-edge $\{1\}$.
We schematically represent $G$ using solid half-lines and lines for half-edges and positive edges, respectively, and using dotted lines for negative edges, as demonstrated in Figure~\ref{fig:signgraph}.
\end{example}

\begin{figure}[ht]
\begin{center}
\includegraphics{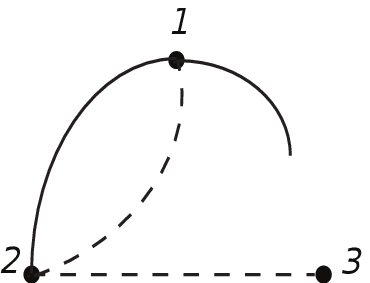}
\caption{}
\label{fig:signgraph}
\end{center}
\end{figure}

\begin{definition}
Let $G$ be a signed graph on $[n]$.  
A \emph{$c$-coloring} of $G$ is a map 
\[
\phi:[n]\rightarrow \{-c,-c+1,\ldots,-1,0,1,\ldots,c\} \, .
\]
A $c$-coloring $\phi$ is \emph{proper} if $\phi(i)\neq \phi(j)$ for all positive edges $\{i,j\}$ in $G$, $\phi(i)\neq -\phi(j)$ for all negative edges $\{i,j\}$ in $G$, and $\phi(i)\neq 0$ for all half-edges $i\in E(G)$.
Denote by $\chi_G(2c+1)$ the number of proper $c$-colorings of $G$.
\end{definition}

\begin{theorem}[Zaslavsky \cite{zas2}]\label{thm:chrompoly}
For $G$ a signed graph on $[n]$, the function $\chi_G(2c+1)$ is given by a polynomial of degree $n$.
\end{theorem}

\begin{example}
For $G$ as in Example~\ref{ex:main}, we have
\[
\chi_G(\lambda)=(\lambda-1)(\lambda-2)(\lambda-1)=\lambda^3-4\lambda^2+5\lambda-2 \, , 
\]
where $\lambda$ is the number of colors in a set of colors containing the color 0.  Note that evaluating $\chi_G(\lambda)$ at $\lambda=2c+1$ yields the number of proper $c$-colorings of $G$.
\end{example}

The key to proving polynomiality of $\chi_G(2c+1)$ is the relation between signed graphs, contractions, and deletions, which we will need subsequently.
Contraction/deletion for signed graphs relies upon the idea of switching a signed graph at a vertex.

\begin{definition}\label{def:switch}
Let $G$ be a signed graph with sign map $\sigma$, and let $v$ be a vertex of $G$.
We say that the signed graph $G'$ is obtained by \emph{switching $G$ at $v$} if the vertex and edge sets for $G'$ are $V(G)$ and $E(G)$, while the sign map $\sigma'$ for $G'$ is given by
\[
\sigma'(\{i,j\}):=\left\{
\begin{array}{ll}
-\sigma(\{i,j\}) & \text{ if } v=i \text{ or } v=j \\
\sigma(\{i,j\}) & \text{ else } 
\end{array}
\right. \, .
\]
If $H$ is obtained from $G$ by a finite sequence of switches, we say that $G$ and $H$ are \emph{switching equivalent}.
\end{definition}

Thus, one switches from $G$ to $G'$ at $v$ by negating the sign on all edges in $G$ incident with $v$.
The following proposition, demonstrating the role played by switching, is simple to prove.

\begin{proposition}\label{prop:chromswitchequiv}
If $G$ and $H$ are switching equivalent, then 
\[
\chi_H(2c+1)=\chi_G(2c+1) \, .
\]
\end{proposition}

\begin{definition}\label{def:contdel}
Let $G$ be a signed graph with an edge $e=\{i,j\}\in E(G)$ with sign $\sigma(e)$.
The \emph{deletion} of $G$ by $e$, denoted $G\setminus e$, is the signed graph obtained by removing $e$ from $E(G)$.
A \emph{contraction} of $G$ by $e$, denoted $G/e$, is a signed graph in the same switching class as the graph obtained from $G$ by the following process (which is well-defined up to switching equivalence).
\begin{itemize}
\item If $\sigma(e)=1$, then delete $e$ from $E(G)$ and contract as in the ordinary graph case by identifying $i$ and $j$ in $V(G)$.
When $\{i,j\}$ is also present in $E(G)$ as a negative edge, add a half-edge at the vertex given by $i=j$ after contracting (if this half-edge is not already present).
\item If $\sigma(e)=-1$, then first switch at an endpoint of $e$ so that $\sigma(e)=1$ and proceed as in the positive edge contraction case.
\end{itemize}

Given a half-edge $i\in E(G)$, the \emph{deletion} of $G$ by $i$, denoted $G\setminus i$, is the signed graph obtained by removing $i$ from $E(G)$.
The \emph{contraction} of $G$ by $i$, denoted $G/i$, is the signed graph with vertex set $V(G)\setminus \{i\}$ and edge set $\{e\setminus \{i\}\mid e\in E(G)\}$.
\end{definition}

Note that $E(G)$ is a multiset, thus it is possible that two copies of $\{i,j\}$ are contained in $E(G)$ with different signs.
If this is the case, then only one copy of $\{j\}$ is retained in the edge set of the deletion and contraction.
The key property of contraction/deletion, and what makes it relevant for the proof of Theorem~\ref{thm:chrompoly}, is given next.

\begin{proposition}
Given a signed graph $G$ with positive edge $e$, 
\[
\chi_G(2c+1)=\chi_{G\setminus e}(2c+1) - \chi_{G/e}(2c+1) \, . 
\]
\end{proposition}

A final fact we need is that when the chromatic polynomial is expressed as
\begin{equation}\label{eqn:cjs}
\chi_G(\lambda)=\sum_{j=1}^{n-1}(-1)^{n-j}c_j(G)\lambda^j + \lambda^n \, ,
\end{equation}
the $c_j$'s are non-negative integers.
This can be seen in several ways, e.g. by recognizing $\chi_G(\lambda)$ as the characteristic polynomial of the arrangement $\mathcal B_G$ defined in the next section.

%%%%%%%%%%%%%%%%%%%%%%%%%%%%%%%%%%%%%%%%%

\section{Signed graphic arrangements, coloring complexes, and group actions}\label{sec:complex}

Our construction of signed graph coloring complexes involves the following hyperplane arrangement.

\begin{definition}
The \emph{type B braid arrangement} is the collection of hyperplanes
\[
\mathcal B_n:= \{H_{ij}^{+1} \mid 1 \leq i < j \leq n \} \cup \{H_{ij}^{-1} \mid 1 \leq i < j \leq n \} \cup \{H_i \mid 1 \leq i \leq n \}
\]
where $H_{ij}^{+1} = \{ (x_1, \hdots, x_n) \in \R^n \mid x_i = x_j \}$, $H_{ij}^{-1} = \{ (x_1, \hdots, x_n) \in \R^n \mid x_i = -x_j \}$, and $H_i = \{ (x_1, \hdots, x_n) \mid x_i = 0 \}$.  
\end{definition}

The arrangement $\mathcal B_n$ induces a regular cell decomposition $\Delta_{\mathcal B_n}$ of the sphere $S^{n-1}$, which we describe using the choice of $\partial[-1,1]^n$ as our preferred representation of $S^{n-1}$.
$\mathcal B_n$ induces a triangulation of $\partial[-1,1]^n$ where each vertex $v = (v_1, \hdots, v_n)$ of the resulting triangulation can be identified with a nonempty subset of $[n] \cup -[n]$ by the rule that for each $v_i$, $\pm i$ is included in the subset if $v_i = \pm 1$, respectively, and $i$ is not included in the subset if $v_i = 0$.
The faces of the triangulation are given by collections of vertices corresponding to chains (with respect to inclusion) of subsets of this type.
Alternatively, given such a chain 
\[
C:= Q_1 \subset Q_2 \subset Q_3\subset \cdots \subset Q_r \, ,
\]
we associate to $C$ the ordered set partition of $[n]\cup -[n]$ given by
\[
Q_1 \, | \, Q_2\setminus Q_1 \, | \, Q_3\setminus Q_2 \, |\cdots| \, Q_r\setminus Q_{r-1} \, | \, [n]\cup[-n] \setminus Q_r \, .
\]
It is clear that $C$ may be fully recovered from its associated partition.

\begin{example}
The triangulation of $\partial[-1,1]^3$ induced by $\mathcal B_3$ is shown in Figure~\ref{fig:B3}.
The chamber marked G has vertices $(0,1,0)$, $(0,1,-1)$, and $(-1,1,-1)$; thus, G is identified with the chain
\[
\{2\}\subset \{2,-3\} \subset \{-1,2,-3\} \, 
\]
with associated partition \((2|-3|-1|1,-2,3)\).
Similarly, the chamber marked B has vertices $(1,0,0)$, $(1,-1,0)$, and $(1,-1,1)$, hence corresponds to the chain
\[
\{1\}\subset \{1,-2\} \subset \{1,-2,3\} \, 
\]
with associated partition \((1|-2|3|-1,2,-3) \).
\end{example}

\begin{figure}[ht]
\begin{center}
\includegraphics{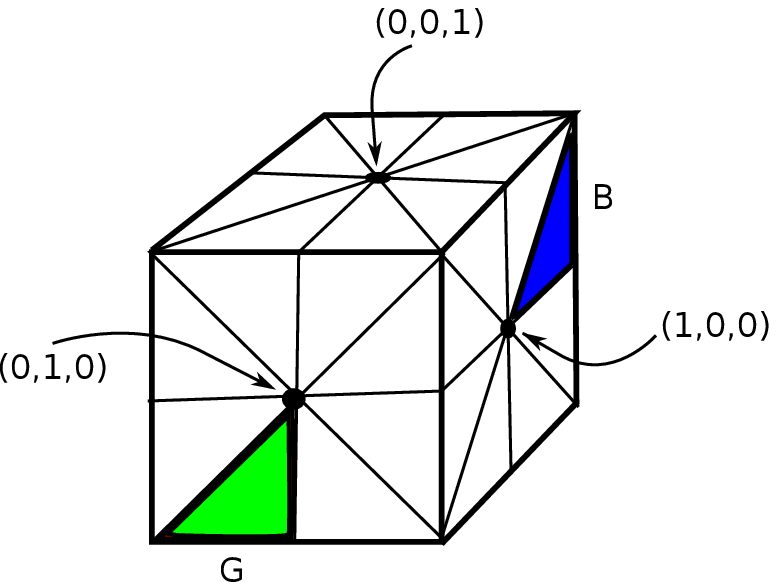}
\caption{}
\label{fig:B3}
\end{center}
\end{figure}

\begin{definition}Given a signed graph $G$ with sign map $\sigma$, the \emph{(signed) graphic arrangement} corresponding to $G$ is the subarrangement $\mathcal B_{G}$ of $\mathcal B_n$ defined by
\[
\mathcal B_{G} = \{ H_{ij}^{\sigma(\{i,j\})} \mid \{i,j\} \in E(G) \} \cup \{H_i : \{i\}\in E(G)\} \, .
\]
\end{definition}

\begin{example}\label{ex:grapharr}
Continuing with the signed graph of Example~\ref{ex:main}, we see that
\[
\mathcal B_{G}=\{(x_1,x_2,x_3)\in \R^3\mid \text{ either }x_1=0, \phantom{.} x_1=x_2, \phantom{.} x_1=-x_2 \text{ or } x_2=-x_3 \} \, .
\]
\end{example}

Before discussing the geometric manifestation of coloring complexes using $\mathcal{B}_G$, we will first define the coloring complex of a signed graph in a purely combinatorial manner, using the viewpoint of ordered set partitions developed above.
For $G$ a signed graph on $[n]$, we say a subset $A\subset [n]\cup -[n]$ \emph{contains an edge} of $G$ if one of the following two cases hold for some pair $\{a,b\}\subset A$.
\begin{itemize}
\item $\{a,b\}$ is a positive edge in $G$, or
\item $a\in [n]$ and $b\in -[n]$, and $\{a,-b\}$ is a negative edge in $G$.
\end{itemize}

\begin{definition}\label{def:colcomplex}
Given a signed graph $G$, the coloring complex $\Delta_G$ is the simplicial complex whose facets are ordered set partitions $P_1|P_2|\cdots |P_{n}$ of $[n]\cup -[n]$ such that 
\begin{enumerate}
\item $P_{n}=\left([n]\cup -[n]\right) \setminus \cup_{j=1}^{n-1} P_j$,
\item for each pair $\{j,-j\}$ with $j\in [n]$, either $j$ or $-j$ is contained in $P_n$, and
\item either there exists a unique non-singleton block $P_i$ with $1\leq i\leq n-1$ that contains an edge of $G$, or
\item for some half-edge $i\in E(G)$, $\{i,-i\}\subset P_n$.
\end{enumerate}  
\end{definition}

The faces of $\Delta_G$ are formed by merging adjacent blocks in the partitions defining the facets.
Thus, the vertices of $\Delta_G$ are given by partitions $P_1|\left([n]\cup -[n]\right) \setminus P_1$ where $P_1$ is obtained by merging an initial segment of blocks in one of the facets of $\Delta_G$ described above.
The $r$-dimensional faces of $\Delta_G$ are the ordered set partitions with $r+2$ blocks.
Note that the role of the empty set is taken by the trivial partition $[n]\cup -[n]$.
As in the case for $\Delta_{\mathcal B_n}$, each partition $P_1|P_2|\cdots |P_n$ corresponds uniquely to a chain in $2^{[n]\cup -[n]}$ of the form
\[
\emptyset \subset P_1 \subset P_1\cup P_2 \subset \cdots \subset \cup_{i=1}^jP_i \subset \cdots \subset [n]\cup -[n] \, .
\]

Geometrically, the coloring complex arises as $\Delta_G=\mathcal{B}_G\cap \partial[-1,1]^n$.
The space $\Delta_G$ inherits the simplicial triangulation described above via the restriction of $\Delta_{\mathcal B_n}$ to $\Delta_G$.
The connection to the triangulation of $\partial[-1,1]^n$ induced by $\mathcal B_n$ is immediate from our previous discussion.
We will freely use the notation $\Delta_G$ to denote both the topological space $\mathcal{B}_G\cap \partial[-1,1]^n$ and the abstract simplicial complex obtained after intersecting with $\Delta_{\mathcal B_n}$.
Given this geometric observation, it follows that $\Delta_G$ is an example of a link complex of a subspace arrangement, resulting in the following theorem.

\begin{theorem}{\rm (Hultman, \cite[Theorem 4.2]{hu})}
For any signed graph $G$, $\Delta_G$ is shellable, hence is homotopy equivalent to a wedge of spheres of dimension $n-2$.
\end{theorem}

One reason the complex $\Delta_G$ is important is that it provides a path through which we can interpret chromatic polynomials as Hilbert polynomials of graded algebras.

\begin{theorem}{\rm (Hultman, \cite[Corollary 5.8]{hu})}
Let $k[\cone(\Delta_G)]$ be the Stanley-Reisner ring of the cone over $\Delta_G$.
The Hilbert polynomial of $k[\cone(\Delta_G)]$ is given by 
\[
F(k[\cone(\Delta_G)];c)=(2c+1)^n-\chi_G(2c+1) \, .
\]
\end{theorem}

\begin{example}
Let $G$ be as in Example~\ref{ex:main}, hence $\mathcal B_G$ as in Example~\ref{ex:grapharr}.
The complex $\Delta_G$, arising as a subcomplex of the triangulation of $\partial[-1,1]^3$ shown in Figure~\ref{fig:B3}, is given in Figure~\ref{fig:colcomplex}.
It is straightforward to verify that $\displaystyle \Delta_G\simeq \vee_{i=1}^{11}S^1$.
\end{example}

\begin{figure}[ht]
\begin{center}
\includegraphics{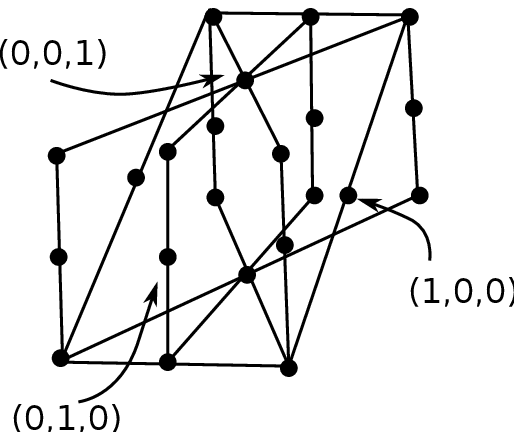}
\caption{}
\label{fig:colcomplex}
\end{center}
\end{figure}

We will later need to undertake a detailed analysis of the boundary operators for $\Delta_G$, for which the following notation is needed.
Let $d_i$ denote the map on the $r$-faces of $\Delta_G$ defined by:
\[
d_i(P_1\mid \cdots \mid P_i \mid P_{i+1} \mid \cdots \mid P_{r+2}) = (P_1 \mid \cdots \mid P_{i-1} \mid P_{i} \cup P_{i+1} \mid P_{i+2} \mid \cdots \mid P_{r+2}).
\]
The boundary operator on the $r$-chains of $\Delta_G$ is then defined by
\[
\partial_r = \sum_{i=1}^{r+1} (-1)^{i-1} d_i .
\]
Note that since $\Delta_G$ is homotopic to a wedge of spheres of dimension $n-2$, the reduced homology of $\Delta_G$ is non-vanishing only in dimension $n-2$.

Another reason that signed graph coloring complexes are interesting is that their chain spaces admit a family of actions by hyperoctahedral groups, which we will now introduce.
All relevant background regarding hyperoctahedral groups can be found in \cite[Section 8.1]{bb}.
Let $B_n$ denote the $n$-th hyperoctahedral group, i.e., $B_n$ is the set of all permutations of $[-n,n]$ such that $\pi(-i)=-\pi(i)$ for all $i\in \{0,1,\ldots,n\}$.
Note that $\pi(0):=0$ for all $\pi\in B_n$.
For an element $\pi\in B_n$, we say the \emph{window} for $\pi$ is 
\[
\pi=[\pi(1) \, \pi(2) \, \cdots \, \pi(n)] \, ;
\]
thus, the window for $\pi$ is analogous to one-line notation for the symmetric group.
Recall that for an element $\pi\in B_n$, the \emph{sign} of $\pi$ is
\[
\sgn(\pi):=(-1)^{\ell(\pi)}
\]
where $\ell(\pi)$ is the length of $\pi$.
Recall also that the \emph{descent statistic} for $\pi\in B_n$ is
\[
\des(\pi):=\#\{i\in [0,\ldots,n-1]\mid \pi(i)>\pi(i+1)\} \, .
\]

Set $P_{-i}:=-P_i$.
For any signed graph $G$, an element $\pi\in B_{r+1}$ acts on $C_r(\Delta_G;\C)$, the space of $r$-chains of the coloring complex for $G$, by extending linearly the action on basis elements given by
\[
\pi (P_1 \mid P_2 \mid \cdots \mid P_{r+1} \mid P_{r+2}) = (P_{\pi^{-1}(1)} \mid \cdots \mid P_{\pi^{-1}(r+1) } \mid P_{r+2}') \, ,
\]
where $P_{r+2}'$ is $P_{r+2}$ with $j\in P_{r+2}$ changed to $-j\in P_{r+2}'$ if $\pi$ changed the sign on the block containing $j$.
Informally, $\pi$ acts by permuting the first $r+1$ blocks in the partitions defining the $r$-faces of $\Delta_G$ by $\pi^{-1}$, changing the signs of all elements in blocks where a sign change occurs on the index, and subsequently modifying $P_{r+2}$ to account for those sign changes.

\begin{example}
Let $\pi=[2 \, -1 \, -3]\in B_3$, hence $\pi^{-1}=[-2 \phantom{..} 1 \, -3]$.
For 
\[
\tau:=(1,3 |-2,5| 6 | -1,2,-3,4,-4,-5,-6)\in C_2(\Delta_{B_6}) \, ,
\]
we have that
\[
\pi(\tau)=(2,-5|1,3|-6|-1,-2,-3,4,-4,5,6) \, .
\]
\end{example}

%%%%%%%%%%%%%%%%%%%%%%%%%%%%%%%%%%%%%%%%%%%%%%%%%%%%%%%%%%%

\section{Eulerian idempotents and type B {H}odge decompositions}\label{sec:eulerian}

Our goal in this section is to prove Theorem~\ref{thm:hodge}, stating that $H_{n-2}(\Delta_G)$ admits a particular type of direct sum decomposition called a \emph{type B Hodge decomposition}.
This decomposition arises from the actions of $B_{r+1}$ on $C_r(\Delta_G)$ defined in the previous section, through the study of the following family of idempotents defined by F. Bergeron and N. Bergeron \cite{bn}. 

\begin{definition}
The \emph{Eulerian idempotents} of type B are elements $\tilde{\rho}_n^{(j)}\in \C[B_n]$ defined by 
\begin{equation}\label{eqn:idemdef}
\tilde{\rho}_n(x)=\sum_{j=0}^nx^j\tilde{\rho}_n^{(j)} := \sum_{\pi \in B_n}\left[\frac{\prod_{k=1}^n(x-2\des(\pi)+2k-1)}{2^nn!} \right] \sgn (\pi)\pi \, .
\end{equation}
\end{definition}

Bergeron and Bergeron \cite[Section 4]{bn} show that
\begin{equation}\label{eqn:idempotent}
\tilde{\rho}_n^{(r)}\tilde{\rho}_n^{(s)}=\delta_{r,s}\tilde{\rho}_n^{(r)} \, .
\end{equation}
Thus, these elements form a family of orthogonal idempotents in $\C[B_n]$.
Evaluating \eqref{eqn:idemdef} at $x=1$ yields 
\begin{equation}\label{eqn:sumtoid}
\sum_{j=0}^n\tilde{\rho}_n^{(j)}=\id
\end{equation}
after observing that
\[
\frac{\prod_{k=1}^n(1-2\des(\pi)+2k-1)}{2^nn!}=\binom{-\des(\pi)+n}{n}=
\left\{
\begin{array}{cl}
1 & \text{ if } \des(\pi)=0 \\
0 & \text{ else }\\
\end{array}
\right. \, .
\]

\begin{theorem}\label{thm:hodge}
Let $C_r^{(j)}(\Delta_G) = \tilde{\rho}_r^{(j)}C_r(\Delta_G)$.  
The chain complex for $\Delta_G$ decomposes as
\[
(C_{*}(\Delta_G), \partial_{*}) = \bigoplus_{j \geq 0} (C_{*}^{(j)}(\Delta_G), \partial_{*}) \, ,
\]
thus
\[
H_r(\Delta_G) = \bigoplus_{j \geq 0} H_{r}^{(j)}(\Delta_G).
\]
\end{theorem}

\begin{proof}
This is a straightforward consequence of \eqref{eqn:idempotent} and \eqref{eqn:sumtoid} along with the relation $\partial_r\tilde{\rho}_{r+1}^{(j)}=\tilde{\rho}_{r}^{(j)}\partial_r$ found in Lemma~\ref{lem:intertwine}.
\end{proof}

To prove Lemma~\ref{lem:intertwine}, we first define
\[
\tilde{l}_n^{(j)}:= (-1)^{j-1}\sum_{\substack{\pi\in B_n \\ \des(\pi)=j}}\sgn(\pi)\pi \phantom{......}\text{and}\phantom{......} \tilde{\lambda}_{n}^{(j)}=\sum_{i=0}^j(-1)^i\binom{n+i}{i}\tilde{l}_{n}^{(j-i)} \, .
\]
It is straightforward to verify that
\begin{equation}\label{eqn:lambdarho}
\tilde{\lambda}_n^{(j)}=(-1)^{j-1}\tilde{\rho}_n(2j+1)
\end{equation}
by evaluating \eqref{eqn:idemdef} and comparing it to these definitions.
To prove our structural result regarding the $\tilde{\rho}_n^{(j)}$'s, it is easier to work directly with the elements $\tilde{l}_n^{(j)}$ and $\tilde{\lambda}_n^{(j)}$, and then apply~\eqref{eqn:lambdarho}.

\begin{lemma}\label{lem:lintertwine}
Let $G$ be a signed graph with vertex set $[n]$, with coloring complex $\Delta_G$ having chain complex $(C_*(\Delta_G), \partial_*)$.
For each $r$ such that $0 \leq r \leq n-2$ and for each $j$ such that $1 \leq j \leq r + 1$,
\[
\partial_r\tilde{l}_{r+1}^{(j)} = \left(\tilde{l}_{r}^{(j)} + \tilde{l}_{r}^{(j-1)}\right)\partial_r \, .
\]
\end{lemma}

\begin{remark}\label{rem:intertwine}
The following proof is similar to the proof of~\cite[Proposition~5.1]{bn} due to Bergeron and Bergeron.
For the sake of completeness, and because the cited proof contains a few confusing typographical errors, we include a proof here.
\end{remark}

\begin{proof}[Proof of Lemma~\ref{lem:lintertwine}]
Setting 
\[
L_{r+1}^{j} := \{ \pi \in B_{r+1} \mid \des(\pi)=j \} \, ,
\]
the strategy is to consider 
\begin{equation}\label{eq:dsum}
d_i\tilde{l}_{r+1}^{(j)}=d_i(-1)^{j-1}\sum_{\pi\in L_{r+1}^j}\sgn(\pi)\pi=(-1)^{j-1}\sum_{\pi\in L_{r+1}^j}\sgn(\pi)d_i\pi
\end{equation}
for each $i=1,\ldots,r+1$.
For some of the elements $d_i\pi\in d_iL_{r+1}^j$, we will produce $\sigma\in L_{r+1}^{j}$ such that $d_i\sgn(\pi)\pi=-d_i\sgn(\sigma)\sigma$, and thus these terms will cancel pairwise in \eqref{eq:dsum}.
For all other elements, we will bijectively map each element in $d_iL_{r+1}^{j}$ to an element appearing as a summand of
\begin{align*}
\left(\tilde{l}_{r}^{(j)} + \tilde{l}_{r}^{(j-1)}\right)\partial_r & = \sum_{i=1}^{r+1}(-1)^{i-1}\left((-1)^{j-1}\sum_{\sigma\in L_r^j}\sgn(\sigma)\sigma + (-1)^{j-2}\sum_{\sigma\in L_r^{j-1}}\sgn(\sigma)\sigma \right)d_i \\
& = \sum_{i=1}^{r+1}(-1)^{i-1}\left((-1)^{j-1}\sum_{\sigma\in L_r^j}\sgn(\sigma)\sigma d_i + (-1)^{j-2}\sum_{\sigma\in L_r^{j-1}}\sgn(\sigma)\sigma d_i \right) \, ,
\end{align*}
showing that $d_i\pi$ corresponds to a term (with correct sign) $\sigma d_{s}$ in a bijection
\begin{equation}\label{eq:bij}
\bigcup_{i=1}^{r+1} d_iL_{r+1}^j \setminus \{\text{pairwise canceling elements}\} \longleftrightarrow \left( \bigcup_{i=1}^{r+1}L_r^jd_i\right) \bigcup \left( \bigcup_{i=1}^{r+1}L_r^{j-1}d_i\right) \, .
\end{equation}
Doing so will yield our desired equality.
As this becomes a lengthy exercise of case-by-case analysis, we will completely prove some of the cases and provide only the setup for the rest.

\textbf{Case:} $i=r+1$.
\smallskip 

Let $I_{r+1}$ denote the element of $B_{r+1}$ that sends $r+1$ to $-(r+1)$ and fixes all other elements; note that $\ell(I_{r+1})$ is odd, as 
\[
I_{r+1}=(r,r+1)\cdots (2,3)(1,2)s_0 (1,2)(2,3)\cdots (r,r+1) \, ,
\]
where $s_0$ is the generator of $B_{r+1}$ sending $1$ to $-1$ and fixing all other elements.
Next observe that $d_{r+1}I_{r+1}=d_{r+1}$, because when the $(r+1)$-st and $(r+2)$-nd blocks of an ordered partition are merged, all the elements of the $(r+1)$-st block appear in the merged block with all possible signs.

Let $\pi\in L_{r+1}^j$ and let $\sigma=I_{r+1}\pi$; it follows from the length of $I_{r+1}$ being odd that $\sgn(\pi)=-\sgn(\sigma)$, by \cite[Proposition~1.4.2~(ii)]{bn}.
Note that $\sigma$ and $\pi$ differ only in that the $\pm(r+1)$ appearing in $\pi$ is negated, but does not change position.
If $\pi_{r+1}\neq \pm (r+1)$, then $\sigma=I_{r+1}\pi$ has the same number of descents as $\pi$, since if $r+1$ is in position $j$, then $\pi$ is forced to have an ascent in position $j-1$ and a descent in position $j$, while if $-(r+1)$ is in position $j$, then $\pi$ is forced to have a descent in position $j-1$ and an ascent in position $j$.
Hence $d_{r+1}\sgn(\sigma)\sigma=d_{r+1}I_{r+1}(-\sgn(\pi))\pi=-d_{r+1}\sgn(\pi)\pi$, and the pair of terms $d_{r+1}\sgn(\sigma)\sigma$ and $d_{r+1}\sgn(\pi)\pi$ cancel each other in \eqref{eq:dsum}.

Suppose now that $\pi_{r+1}=r+1$.
Setting $\sigma=\pi_1\pi_2\cdots \pi_r$, it is straightforward to verify that
\begin{itemize}
\item $\sigma\in B_{r}$, 
\item $d_{r+1}\pi=\sigma d_{r+1}$, 
\item $\des(\sigma)=\des(\pi)$, and
\item $\sgn(\pi)=\sgn(\sigma)$.
\end{itemize}
On the other hand, if $\pi_{r+1}=-(r+1)$, then setting $\sigma=\pi_1\pi_2\cdots \pi_r$ it is again straightforward to verify that
\begin{itemize}
\item $\sigma\in B_{r}$,
\item $d_{r+1}\pi=\sigma d_{r+1}$,
\item $\des(\sigma)=\des(\pi)-1$, and 
\item$\sgn(\sigma)=-\sgn(\pi)$.
\end{itemize}
Mapping $d_{r+1}\pi$ to $\sigma d_{r+1}$ in our correspondence yields a bijection (with correct signs) pairing an element $d_{r+1}\pi\in d_{r+1}L_{r+1}^j$ satisfying $\pi_{r+1}=\pm(r+1)$ with an element of $L_r^jd_{r+1}\cup L_{r}^{j-1}d_{r+1}$, leading to the equality
\[
d_{r+1}\tilde{l}_{r+1}^{(j)} = \left(\tilde{l}_{r}^{(j)} + \tilde{l}_{r}^{(j-1)}\right)d_{r+1} \, .
\]

\textbf{Case:} $1 \leq i \leq r$.  
\smallskip

For each $i$ and each $\pi\in L_{r+1}^j$, the relative position of $\pm i$ and $\pm(i+1)$ in the window for $\pi$ determines how $\pi$ is handled.
There are five situations that can occur:
\begin{itemize}
\item $\pi^{-1}(i+1) = \pi^{-1}(i) + 1 > 0$, implying that $\pi=[\cdots i\, (i+1)\cdots]$
\item $\pi^{-1}(i+1) = \pi^{-1}(i) + 1 < 0$, implying that $\pi=[\cdots -(i+1)\, -i \cdots]$
\item $\pi^{-1}(i+1) = \pi^{-1}(i) - 1 > 0$, implying that $\pi=[\cdots (i+1)\, \,  i\cdots]$
\item $\pi^{-1}(i+1) = \pi^{-1}(i) - 1 < 0$, implying that $\pi=[\cdots -i\, \,  -(i+1)\cdots]$
\item $\pi^{-1}(i+1) \neq \pi^{-1}(i) \pm 1$, containing all remaining cases.
\end{itemize}
\smallskip

We sketch below how to assign to each $d_i\pi$ a unique $\sigma d_s$ in each of these cases, and provide at the end a proof that these assignments are bijective as claimed in~\eqref{eq:bij}.

\textbf{Subcase:} Suppose $\pi^{-1}(i+1) = \pi^{-1}(i) + 1 > 0$, implying that $\pi=[\cdots i\, (i+1)\cdots]$.
\smallskip

Define $s$ as the index such that $\pi_s = i$.
We want to associate to $\pi$ a unique $\sigma\in L_r^j$ with the following properties:
\begin{itemize}
  \item $\des(\sigma) = \des(\pi)$,
  \item $\sgn(\sigma) = (-1)^{i-s}\sgn(\pi)$, and
  \item $d_i \pi = \sigma d_s$.
\end{itemize}
We claim that these properties are satisfied by $\sigma = \sigma_{1} \hdots \sigma_{s} \sigma_{s+2} \hdots \sigma_{r+1}$ where
\begin{itemize}
\item $\sigma_{m} = \pi_{m}$ if $|\pi_{m}| < i+1$ 
\item $|\sigma_{m}| = |\pi_m | -1$ if $|\pi_m| > i + 1$, and
\item the sign pattern for $\sigma$ is the same as that for $\pi$, i.e. if $\pi_j<0$, then $\sigma_j$ is also negative.
\end{itemize}

To prove that $\des(\sigma)=\des(\pi)$, note that all pairwise inequality relationships are preserved between $\pi$ and $\sigma$; thus, the only possible position of an additional descent in $\pi$ that does not occur in $\sigma$ is between $\pi_s=i$ and $\pi_{s+1}=i+1$, where no descent occurs.

To show that $\sgn(\sigma) = (-1)^{i-s}\sgn(\pi)$, observe that there exist $i-s$ adjacent transpositions $t_1,\ldots,t_{i-s}$ in $B_{r+1}$ such that $\pi t_1\cdots t_{i-s}$ has $i+1$ appearing in window position $i+1$; we do so by exchanging adjacent entries in the window for $\pi$ repeatedly to bring $i+1$ from position $s+1$ to position $i+1$.
Then, $\sigma$ is obtained from $\pi t_1\cdots t_{i-s}$ by deleting position $i+1$ and lowering the label for all window elements greater than $i+1$.
Thus, both $\sigma$ and $\pi t_1\cdots t_{i-s}$ can be expressed as a product of the same number of adjacent transpositions and, for each negative element appearing in the window for $\pi$, an odd number of hyperoctahedral group Coxeter generators.
Thus, the length of these two elements have the same parity, and our result follows.

Finally, to show $d_i \pi = \sigma d_s$ it suffices to show that the $m$-th block in the image of $(P_1 \mid \cdots \mid P_{r+1} \mid P_{r+2})$ is the same under $d_i \pi$ and $\sigma d_s$.  
First, suppose that $|\pi_m| = k < i $.  
Then $d_i \pi$ will map $P_m$ to the $k$-th block location in the image if $\pi_m > 0$, or to $P_{-m}$ in the $k$-th block location in the image if $\pi_m < 0$.  
If $m < s$, then $P_m$ will still be the $m$-th block in the image after $d_s$ is applied.  
Since $k < i $, $\sigma_m = \pi_m$, so $\sigma$ will map $P_m$ to the $k$-th block location in the image if $\pi_m > 0$, or to $P_{-m}$ in the $k$-th block location in the image if $\pi_m < 0$.  
If $m > s+1$, then $P_m$ will be in the $(m-1)$-st location after $d_s$ is applied.  
Notice though that by the definition of $\sigma$, this implies that $\sigma_m$ is in the $(m-1)$-st position of $\sigma$.  
Since $k < i $, $\sigma_m = \pi_m$, and thus $\sigma$ will map $P_m$ to the $k$-th block location in the image if $\pi_m > 0$, or to $P_{-m}$ in the $k$-th block location in the image if $\pi_m < 0$.  

Now suppose that $|\pi_m| = k > i+1 $.  
Then $d_i \pi$ will map $P_m$ to the $(k-1)$-st block in the image if $\pi_m > 0$ or to $P_{-m}$ in the $(k-1)$-st block in the image if $\pi_m < 0$.
If $m < s$, then $P_m$ will still be the $m$-th block in the image after $d_s$ is applied.  
Since $k > i+1 $, $|\sigma_m| = |\pi_m| - 1$, so $\sigma$ will map $P_m$ to the $(k-1)$-st block location in the image if $\pi_m > 0$, or to $P_{-m}$ in the $(k-1)$-st block location in the image if $\pi_m < 0$.  
If $m > s+1$, then $P_m$ will be in the $(m-1)$-st location after $d_s$ is applied.  
Notice though that by the definition of $\sigma$, this implies that $\sigma_m$ is in the $(m-1)$-st position of $\sigma$.  
Since $k > i + 1 $, $|\sigma_m| = |\pi_m| - 1$, and thus $\sigma$ will map $P_m$ to the $(k-1)$-st block location in the image if $\pi_m > 0$, or to $P_{-m}$ in the $(k-1)$-st block location in the image if $\pi_m < 0$.  

Now suppose that $m = s$.  Then $d_i \pi$ will map $P_s$ to $P_s \cup P_{s+1}$ in the $i$-th location.  
$d_s$ will map $P_s$ to $P_s \cup P_{s+1}$, and since $P_s \cup P_{s+1}$ is in the $s$-th block, $\sigma d_s$ also maps $P_s$ to $P_s \cup P_{s+1}$ in the $i$-th block.
\smallskip

\textbf{Subcase:} Suppose $\pi^{-1}(i+1) = \pi^{-1}(i) + 1 < 0$, implying that $\pi=[\cdots -(i+1)\, -i \cdots]$.
\smallskip

Let $s$ be defined as the index such that $\pi_s = -i-1$.  
It suffices to show that $\sigma$ defined as follows satisfies $\des(\sigma) = \des(\pi)$, $\sgn(\sigma) = (-1)^{i-s} \sgn(\pi)$, and $d_i \pi = \sigma d_s$.
Let $\sigma = \sigma_{1} \hdots \sigma_{s} \sigma_{s+2} \hdots \sigma_{r+1}$ where 
\begin{itemize}
\item $\sigma_{m} = \pi_{m}$ if $|\pi_{m}| < i$ 
\item $|\sigma_{m}| = |\pi_m | -1$ if $|\pi_m| > i $, and
\item the sign pattern for $\sigma$ is the same as that for $\pi$, i.e. if $\pi_j<0$, then $\sigma_j$ is also negative. 
\end{itemize}
Thus, $\sigma d_s$ is the element in $L_r^jd_s$ uniquely paired with $d_i\pi$.
This argument is similar to the previous subcase.

\smallskip

\textbf{Subcase:} Suppose $\pi^{-1}(i+1) = \pi^{-1}(i) - 1 > 0$, implying that $\pi=[\cdots (i+1)\, \,  i\cdots]$.
\smallskip

Let $s$ be defined as the index such that $\pi_s = i + 1$.  
It suffices to show that $\sigma$ defined as follows satisfies, $\des(\sigma) = \des(\pi) - 1$, $\sgn(\sigma) = (-1)^{i-s+1} \sgn(\pi)$, and $d_i \pi = \sigma d_s$. 
Let $\sigma = \sigma_{1} \hdots \sigma_{s-1} \sigma_{s+1} \hdots \sigma_{r+1}$ where 
\begin{itemize}
\item $\sigma_{m} = \pi_{m}$ if $|\pi_{m}| < i + 1$,
\item $|\sigma_{m}| = |\pi_m | -1$ if $|\pi_m| > i + 1$, and
\item the sign pattern for $\sigma$ is the same as that for $\pi$, i.e. if $\pi_j<0$, then $\sigma_j$ is also negative. 
\end{itemize}
Thus, $\sigma d_s$ is the element of $L_r^{j-1} d_s$ uniquely paired with $d_i\pi$.
This argument is similar to the previous subcase.
Note that the property $\sgn(\sigma) = (-1)^{i-s+1} \sgn(\pi)$ is necessary because while $\pi\in L_{r+1}^j$, we have that $\sigma\in L_{r}^{j-1}$ and thus in $\tilde{l}_r^{(j-1)}$ we have that $\sgn(\sigma)\sigma$ is multiplied by $(-1)^{j-1}$ rather than $(-1)^j$.

\smallskip

\textbf{Subcase:} Suppose $\pi^{-1}(i+1) = \pi^{-1}(i) - 1 < 0$, implying that $\pi=[\cdots -i\, \,  -(i+1)\cdots]$.
\smallskip

Let $s$ be the index such that $\pi_s = -i $.  
It suffices to show that $\sigma$ defined as follows satisfies $\des(\sigma) = \des(\pi) - 1$, $\sgn(\sigma) = (-1)^{i-s+1} \sgn(\pi)$, and $d_i \pi = \sigma d_s$.
Let $\sigma = \sigma_{1} \hdots \sigma_{s-1} \sigma_{s+1} \hdots \sigma_{r+1}$ where 
\begin{itemize}
\item $\sigma_{m} = \pi_{m}$ if $|\pi_{m}| < i $,
\item $|\sigma_{m}| = |\pi_m | -1$ if $|\pi_m| > i $, and
\item the sign pattern for $\sigma$ is the same as that for $\pi$, i.e. if $\pi_j<0$, then $\sigma_j$ is also negative. 
\end{itemize}
Thus, $\sigma d_s$ is the element of $L_r^{j-1}d_s$ uniquely paired with $d_i\pi$.
This argument is similar to the previous subcase, and the same comment as in the previous subcase about $\sgn(\sigma) = (-1)^{i-s+1} \sgn(\pi)$ applies.

\smallskip

\textbf{Subcase:} Suppose $\pi^{-1}(i+1) \neq \pi^{-1}(i) \pm 1$.
\smallskip

Setting $\sigma = (i, i+1)\pi$, one can show that $\sgn(\pi)=-\sgn(\sigma)$, $\des(\pi)=\des(\sigma)$, and that $d_i\pi=d_i\sigma$.
Thus, the terms corresponding to $d_i\pi$ and $d_i\sigma$ cancel in \eqref{eq:dsum}.

\smallskip

\textbf{Unique bijection:} 
To see that our correspondences above are bijective, consider an element $\sigma d_s\in L_r^{j}d_s$, where $\sigma=\sigma_1\cdots \sigma_{s-1}\sigma_s\sigma_{s+1}\cdots \sigma_r$.
Then $\sigma d_s$ is obtainable from some $d_j\pi$ via our above process if $\sigma$ was obtained by deleting the $s$-th or $(s+1)$-st element from $\pi$, i.e.
\[
\text{Case A: } \sigma=[\sigma_1\cdots \sigma_s\underbrace{\phantom{..}}_{\substack{\pi_{s+1}\\ \text{dropped}\\ \text{here}}}\sigma_{s+1}\cdots \sigma_r] \text{  or  Case B: } \sigma=[\sigma_1\cdots \sigma_{s-1}\underbrace{\phantom{..}}_{\substack{\pi_{s}\\ \text{dropped}\\ \text{here}}}\sigma_{s}\cdots \sigma_r] \, .
\]
Considering our claimed bijective map described above, in both Case A and Case B it is the element $\sigma_s$ that determined which element of $\pi$ was dropped.
Suppose that $\sigma_s=\pm i$.
For each of Case A and Case B, a fixed parity for $\sigma_s$ yields a unique $\pi,j$ such that $d_j \pi$ maps to $\sigma d_s$ under our map, and hence our claimed bijection~\eqref{eq:bij} is established.
\end{proof}

\begin{lemma}\label{lem:intertwine}
\[
\partial_r\tilde{\lambda}_{r+1}^{(j)}=\tilde{\lambda}_{r}^{(j)}\partial_r \phantom{.....}\text{and}\phantom{......}\partial_r\tilde{\rho}_{r+1}^{(j)}=\tilde{\rho}_{r}^{(j)}\partial_r
\]
\end{lemma}

\begin{proof}
Since 
\[
\tilde{\lambda}_{r+1}^{(j)}=\sum_{i=0}^j(-1)^i\binom{r+1+i}{i}\tilde{l}_{r+1}^{(j-i)} \, ,
\]
it follows from Lemma~\ref{lem:lintertwine} that
\begin{align*}
\partial_r\tilde{\lambda}_{r+1}^{(j)} & = \partial_r\sum_{i=0}^j(-1)^i\binom{r+1+i}{i}\tilde{l}_{r+1}^{(j-i)} \\
& = \sum_{i=0}^j(-1)^i\binom{r+1+i}{i}\left( \tilde{l}_{r}^{(j-i)}+\tilde{l}_{r}^{(j-i-1)} \right)\partial_r \\
& = \left[ (-1)^0\binom{r+1}{0}\left(\tilde{l}_{r}^{(j)}+\tilde{l}_{r}^{(j-1)}\right)+(-1)^1\binom{r+2}{1}\left(\tilde{l}_{r}^{(j-1)}+\tilde{l}_{r}^{(j-2)}\right)+ \cdots \right. \\
& \hspace{7mm} \left.  +(-1)^{j-1}\binom{r+1+j-1}{j-1}\left(\tilde{l}_{r}^{(1)}+\tilde{l}_{r}^{(0)}\right)+(-1)^{j}\binom{r+1+j}{j}\tilde{l}_{r}^{(0)}\right]\partial_r\\
& = \left[ (-1)^0\binom{r+1}{0}\tilde{l}_{r}^{(j)} + \left( (-1)^0\binom{r+1}{0}+(-1)^1\binom{r+2}{1}\right)\tilde{l}_{r}^{(j-1)} + \cdots \right. \\
& \hspace{7mm} \left.  +\left((-1)^{j-1}\binom{r+1+j-1}{j-1}+(-1)^j\binom{r+1}{j}\right)\tilde{l}_{r}^0 \right] \partial_r \\
& = \sum_{i=0}^j(-1)^i\binom{r+1}{i}\tilde{l}_{r}^{(j-i)} \partial_r \\
& = \tilde{\lambda}_{r}^{(j)}\partial_r \, ,
\end{align*}
which establishes the first claim.

For the second claim, note that due to the relation
\[
\tilde{\lambda}_{r+1}^{(j)}=(-1)^j\tilde{\rho}_{r+1}(2j+1) \, ,
\]
it follows that the $\tilde{\lambda}_{r+1}^{(j)}$'s and the $\tilde{\rho}_{r+1}^{(j)}$'s are related by an invertible Vandermonde matrix.
Changing basis in this manner from the first to the second set of elements establishes the second claim.
\end{proof}

%%%%%%%%%%%%%%%%%%%%%%%%%%%%%%%%%%%%%%%%%%%%%%%%%%%%%%%%%%

\section{Chromatic polynomial coefficients and {H}odge decompositions}\label{sec:chromatic}

In this section we establish that the coefficients of $\chi_G(\lambda)$ encode (up to sign) the dimensions of the Hodge components for $H_{n-2}(\Delta_G)$.
First we must establish that Hodge decompositions are preserved by switching.

\begin{lemma}\label{lem:hodgeswitch}
If two signed graphs $G$ and $H$ are switching equivalent, then there is a chain complex isomorphism between $C_*(\Delta_G)$ and $C_*(\Delta_H)$ respecting the Hodge decompositions.
\end{lemma}

\begin{proof}
Suppose that $H$ is obtained from $G$ by switching at vertex $i$.
It is straightforward to check that the map $f_i:C_*(\Delta_G)\rightarrow C_*(\Delta_H)$ obtained by exchanging $i$ and $-i$ in every face of $\Delta_G$ is a chain complex isomorphism; one way to see this is to recognize that $\Delta_G=B_G\cap \partial[-1,1]^n$ is taken to $\Delta_H=B_H\cap \partial[-1,1]^n$ by the map $x_i\to -x_i$, and this induces the map $f_i$ at the level of chain complexes.
Using the combinatorial description of coloring complexes given in Definition~\ref{def:colcomplex} and the action of $B_{r+1}$ on $C_r$ defined for any coloring complex, it is immediate that for any $\pi\in B_{r+1}$ we have
\[
\pi \circ f_i = f_i \circ \pi  \, ,
\]
hence
\[
\tilde{\rho}_{r+1}^{(j)}\circ f_i = f_i \circ \tilde{\rho}_{r+1}^{(j)} \, .
\]
Our lemma follows by combining this with the fact that $f_i$ is a chain complex isomorphism.
\end{proof}

\begin{theorem}\label{thm:chromhodge}
Let $G$ be a signed graph on $[n]$ with at least one edge or half-edge.
Writing
\[
\chi_G(\lambda)=\lambda^n+\sum_{j=0}^{n-1}(-1)^{n-j}c_j\lambda^j \, ,
\]
we have $\dim H_{n-2}^{(j)}(\Delta_G)=c_j$.
Equivalently,
\[
(-1)^n\left[ \chi_G(-\lambda)-(-\lambda)^n\right] =\sum_{j=0}^nc_j\lambda^j=\sum_{j=0}^n\dim H_{n-2}^{(j)}(\Delta_G)\lambda^j \, .
\]
\end{theorem}

\begin{proof}
We go by induction on $n$, similar to the proof given by Hanlon~\cite[Theorem 4.1]{ha}.  

\textbf{Base Case:}
Suppose first that $E$ consists of a single half-edge; without loss of generality, we can consider this half-edge to be $\{n\}$.
Then $\Delta_G\cong S^{n-2}$, so $\dim H_{n-2}(\Delta_G)=1$.
Let $\gamma=(1|2|\cdots |n-1|-1 \, -2\, \cdots -n\, n)$; let 
\[
\Gamma:=\left[\frac{1}{2^{n-1}(n-1)!}\sum_{\sigma\in B_{n-1}}\sgn(\sigma)\sigma \right] \gamma \, .
\]
\emph{Claim:} $\partial\Gamma=0$.
Considering the application of each $d_i$ independently, we obtain
\begin{align*}
& \phantom{=} \sum_{i=1}^{n-1}(-1)^{i-1}d_i\cdot \Gamma = \sum_{i=1}^{n-1}\frac{1}{2^{n-1}(n-1)!}\sum_{\sigma\in B_{n-1}}(-1)^{i-1}d_i\sgn(\sigma)\sigma\gamma \\
& = \sum_{i=1}^{n-1}\frac{1}{2^{n-1}(n-1)!}\sum_{\sigma\in B_{n-1}}(-1)^{i-1}d_i\sgn(\sigma)(\gamma_{\sigma^{-1}(1)}|\cdots |\gamma_{\sigma^{-1}(n-1)}|\gamma_n') \\
& = \sum_{i=1}^{n-1}\frac{1}{2^{n-1}(n-1)!}(-1)^{i-1}\sum_{\sigma\in B_{n-1}}\sgn(\sigma)(\gamma_{\sigma^{-1}(1)}|\cdots  |\gamma_{\sigma^{-1}(i)}\cup\gamma_{\sigma^{-1}(i+1)}|\cdots |\gamma_{\sigma^{-1}(n-1)}|\gamma_n') \, .
\end{align*}
For $i\neq n-1$, on the terms of the sum $\sum_{\sigma\in B_{n-1}}\sgn(\sigma)d_i\sigma$, consider the involution $\sigma\to (i,i+1)\sigma$.
This yields a sign-reversing involution on the summands in the final displayed line above.

On the terms of the corresponding sum for $d_{n-1}$, consider the involution $\sigma\to I_{n-1}\sigma$, where
\[
I_{n-1}=(n-2,n-1)\cdots (2,3)(1,2)s_0(1,2)(2,3)\cdots (n-2,n-1) \, .
\]
This yields another sign-reversing involution on the summands in the final displayed line above, hence
\[
\partial\Gamma=0 \, .
\]
Since $\Gamma=\tilde{\rho}_{n-1}^{(n-1)}$, it follows that
\[
\dim H_{n-2}^{(n-1)}(\Delta_G)=1 \, .
\]
Since $\chi_G(\lambda)=(\lambda-1)\lambda^{n-1}=\lambda^n-\lambda^{n-1}$, our base case holds.

In the case where $E$ consists of only an edge, we can without loss of generality consider the edge to be $\{n-1,n\}$.
If we set $\gamma=(1|2|\cdots |\{n,n-1\}|-1\, -2\, \cdots -n)$, then the same analysis as given above holds, establishing this base case as well.

\textbf{Induction:}
Let $G$ be a signed graph with $n \geq 2$ edges, and assume by way of induction that Theorem ~\ref{thm:chromhodge} holds for any signed graph with fewer than $n$ edges.
Let $e$ be an edge of $G$; without loss of generality, we may assume $e$ is either a half-edge or a positive edge, since by Proposition~\ref{prop:chromswitchequiv} and Lemma~\ref{lem:hodgeswitch} we may switch a negative $e$ to obtain a new signed graph with the same chromatic polynomial and Hodge structure.
%For $e = \{ a, b \}$ or $e = \{a, -b\}$, we say that the chain $P = (P_1| \cdots | P_{r+2})$ \emph{contains e} if $\{ a, b \} \subseteq P_i$ or $\{a, -b\} \subseteq P_i$ for some $i$, $1 \leq i \leq r+1$, respectively.  
%For $e = j$, a half-edge, we say that the chain $P= (P_1| \cdots | P_{r+2})$ \emph{contains e} if $j$ and $-j$ are contained in $P_{r+2}$.
Let $E$ be the graph with vertex set $V(G)$ and edge set $\{e\}$.  Let $C_r(E)$ denote the space spanned by the chains $P = (P_1| \cdots |P_{r+2})$ that contain $e$.  Let $D_r$ denote the space spanned by the chains $P = (P_1| \cdots |P_{r+2})$ in $\Delta_G$ that do not contain $e$.  
It follows that
\[
C_r(G) = D_r \oplus C_r(E).
\]
Notice that the action of $B_{r+2}$ commutes with the isomorphism, and thus,
\begin{equation}\label{eqn:first}
C_r^{(j)}(G) = D_{r}^{(j)} \oplus C_r^{(j)}(E) \, .
\end{equation}

Considering $D_r$ next, we claim that 
\[
D_r \simeq C_r(G \backslash e)/(C_r(G \backslash e) \cap C_r(E)) \, .
\]
To prove this, observe that $C_r(G \backslash e)$ has as a spanning set the set of chains $P = (P_1| \cdots |P_{r+2})$ where at least one of the $P_i$ contains an edge of the graph $G \backslash e$.  
This spanning set consists of chains $P$ that contain an edge of the graph $G \backslash e$ as well as the edge $e$, and it consists of chains $P$ that contain an edge of the graph $G \backslash e$ but do not also contain the edge $e$.  
The set of all chains $P$ that contain an edge of $G \backslash e$ but not the edge $e$ form a spanning set for $C_r(G \backslash e)/(C_r(G \backslash e) \cap C_r(E))$.  
Notice that this set is also a spanning set for $D_r$, and the isomorphism then follows.  
Since the action $B_{r+2}$ commutes with this isomorphism, we have
\begin{equation}\label{eqn:second}
D_r^{(j)} \simeq C_r^{(j)}(G \backslash e)/(C_r^{(j)}(G \backslash e) \cap C^{(j)}_r(E)). \, .
\end{equation}

We finally claim that 
\begin{equation}\label{eqn:third}
C_r^{(j)}(G \backslash e) \cap C_r^{(j)}(E)\simeq C_r^{(j)}(G/e) \, .
\end{equation}
Suppose first that $e=\{a,b\}$ is a positive edge.
Consider the map sending 
\[
P=(P_1| P_2| \cdots |P_{r+1}|P_{r+2})\in C_r(G \backslash e) \cap C_r(E)
\]
to the chain $Q = (Q_1 | \cdots | Q_{r+2})\in C_{r}(G/e)$ obtained by replacing the pair $\{a,b\}$ by the symbol $a=b$ representing the contracted vertex in $G/e$.
It is straightforward that this map gives a bijection inducing our desired isomorphism between $C_r(G \backslash e) \cap C_r(E)$ and $C_r(G/e)$, because the pair of symbols $\{a,b\}$ in any basis chain $P\in C_r(G \backslash e) \cap C_r(E)$ is simply replaced by the contracted vertex symbol $a=b$.
Note that surjectivity follows since any chain $Q\in C_{r}(G/e)$ will contain an edge of $G/e$ in some block, which will by definition correspond to an edge in $G\backslash e$, and hence a preimage under our map may be found.
It is clear that this map is invariant under the hyperoctahedral group action, as $a$ and $b$ always are moved as part of the same block in both settings.

Next, suppose that $e = \{j\}$ is a half-edge, and observe that $C_r(E)$ is spanned by chains with $j,-j$ in $P_{r+2}$.  
Consider the map which takes a chain $(P_1|...|P_{r+2})\in C_r(G \backslash e) \cap C_r(E)$ and deletes the $j,-j$ in $P_{r+2}$ to obtain a chain $(P_1|...|P_{r+2}\ \{j,-j\})\in C_r(G/e)$.
We claim that this is a bijection inducing our desired isomorphism.
To prove this, note that for any chain $(Q_1|...|Q_{r+2})\in C_r(G/e)$ we can add $\{j,-j\}$ to $Q_{r+2}$ and obtain a new chain.  
This is actually a chain in the spanning set for $C_r(G\backslash e) \cap C_r(E)$; that it is in $C_r(E)$ is clear. 
To see that it is in $C_r(G \backslash e)$, we consider two possible cases.  
First, if a block $Q_k$ contains an edge in $G$ not incident to $j$, in which case this is also an edge in $G\backslash e$, our chain is a spanning element of $C_r(G \backslash e)$ and we are done.  
Second, if no block $Q_k$ contains an edge in $G/e$, then for some $i$ where $i,j$ is an edge of $G$, we must have that $i,-i$ is in $Q_{r+2}$.  
Then when we add in $j,-j$ to $Q_{r+2}$, we have all of $i,j,-i,-j$ in $Q_{r+2}$.  
This implies that $(Q_1|...|Q_{r+2}\cup \{j,-j\})$ can be obtained by merging the last two blocks in $(Q_1|...|Q_{r+1} | \{i,j\} | (Q_{r+2}\setminus \{i\}) \cup\{-j\})$, and this longer chain corresponds to a spanning element of $C_{r+1}(G \backslash e)$.  
Thus, $(Q_1|...|Q_{r+2}\cup \{i,-i\})$ must also be in the spanning set $C_r(G \backslash e)$.
It is immediate that these maps are invariant under the hyperoctahedral group action, as the final block containing $j,-j$ is always fixed by the group.

From \eqref{eqn:first}, \eqref{eqn:second}, and~\eqref{eqn:third}, it follows that
\[
\dim(C_r^{(j)}(G)) = \dim(C_r^{(j)}(G \backslash e)) - \dim(C_r^{(j)}(G/e)) + \dim(C_r^{(j)}(E)) \, .
\]

Using the fact that the reduced homology $H_*(\Delta_G)$ is only nonvanishing in top dimension, along with the Euler-Poincare identity and our inductive hypothesis, we conclude that 
\begin{align*}
 \phantom{.} & \sum_{j} \dim(H_{n-2}^{(j)}(G)) \lambda^j \\
= \phantom{.} & \sum_{j,r}(-1)^{(n-2)-r} \dim(H_r^{(j)}(G)) \lambda^{j} \\
= \phantom{.} & \sum_{j,r}(-1)^{n-2-r} \dim(C_r^{(j)}(G)) \lambda^j \\
= \phantom{.} & \sum_{j,r}(-1)^{n-2-r} (\dim(C_r^{(j)}(G\backslash e)) - \dim(C_r^{(j)}(G/e)) + \dim(C_r^{(j)}(E))\lambda^j\\
= \phantom{.} & \sum_{j} \dim(H_{n-2}^{(j)}(G\backslash e)) \lambda^j  - (-1)\sum_{j} \dim(H_{n-1}^{(j)}(G/e)) \lambda^j  + \sum_{j} \dim(H_{n-2}^{(j)}(E)) \lambda^j \\
= \phantom{.} & (-1)^n\left[(\chi_{G \backslash e}(-\lambda) - (-\lambda)^n)-(\chi_{G/e}(-\lambda) - (-\lambda)^{n-1})+(\chi_{E}(-\lambda)-(-\lambda)^n)\right] \\
= \phantom{.} & (-1)^n\left[\chi_{G \backslash e}(-\lambda) - \chi_{G/e}(-\lambda) - (-\lambda)^n\right] \\
= \phantom{.} & (-1)^n(\chi_{G}(-\lambda) - (-\lambda)^n). \\
\end{align*}
\end{proof}

%%%%%%%%%%%%%%%%%%%%%%%%%%%%%%%%%%%%%%%%%%%%%%%%%%%%%%%%%%

\end{document}